\documentclass[11pt,letterpaper,leqno]{article}
\usepackage[T1]{fontenc}
\usepackage[utf8]{inputenc}
\usepackage{lmodern}
\usepackage[margin=1in]{geometry}
\usepackage{amssymb,mathtools,bm,amsthm}
\usepackage{booktabs,longtable,array,enumitem}
\usepackage{graphicx,placeins,microtype,needspace}
\usepackage[hidelinks]{hyperref}
\allowdisplaybreaks[2]
\setlist{itemsep=2pt,topsep=4pt}
\newtheorem{theorem}{Theorem}[section]
\newtheorem{lemma}[theorem]{Lemma}
\newtheorem{corollary}[theorem]{Corollary}
\newtheorem{proposition}[theorem]{Proposition}

\theoremstyle{remark}

\newtheorem{example}[theorem]{Example}
\newcommand{\email}[1]{\href{mailto:#1}{\nolinkurl{#1}}}
\newenvironment{keywords}{\par\smallskip\noindent\textbf{Keywords.}\ }{\par\smallskip}
\newenvironment{MSCcodes}{\par\smallskip\noindent\textbf{MSC 2020.}\ }{\par\smallskip}

\numberwithin{equation}{section}

\title{Strong Duality Thresholds and Minimum Curvature for the Nearest Rectangular Symplectic Matrix Problem}
\author{Yulin Xie\thanks{State Key Laboratory of Mathematical Sciences,
Academy of Mathematics and Systems Science, Chinese Academy of Sciences,
Beijing 100190, China; and School of Mathematical Sciences,
University of Chinese Academy of Sciences, Beijing 100049, China.
Yulin Xie (\email{xieyulin@amss.ac.cn}); Yifa Tang, corresponding author
(\email{tyf@lsec.cc.ac.cn}).}
\and Aiqing Zhu\thanks{Department of Mathematics, National University of Singapore,
10 Lower Kent Ridge Road, Singapore 119076 (\email{zaq@nus.edu.sg}).}
\and Yifa Tang\footnotemark[1]}
\hypersetup{pdftitle={Strong Duality Thresholds and Minimum Curvature for the Nearest Rectangular Symplectic Matrix Problem},pdfauthor={Yulin Xie, Aiqing Zhu, and Yifa Tang}}
\date{}
\begin{document}
\maketitle
\begin{abstract}
We study the nearest rectangular symplectic matrix problem in the squared Frobenius norm for arbitrary targets $T\in\mathbb R^{2n\times2k}$, with $n\ge k$.
We prove that its semidefinite relaxation is exact for every target if and only if $n\ge2k-1$.
Under this condition, the primal problem, the relaxation, and both Lagrangian dual problems attain the same optimal value.
In this range, we derive the exact minimum Rayleigh quotient of the Riemannian Hessian at every feasible point under the Euclidean induced metric, and prove that all second-order stationary points are global minimizers.
We also obtain global objective error bounds.
We extend the strong-duality and minimum-curvature results to Hermitian positive definite row weights, proving that strong duality holds for every target if and only if the lowest eigenvalue of the weight has complex multiplicity at least $2k-1$.
The dual formulation leads to a method whose main iterations use matrices of order $2k$, with $O(k^2)$ iterative storage apart from the full input and output.
Across six tested sizes, mean-time speedups are $2.55$--$12.60$ over a Riemannian first-order method and $4.61$--$43.93$ over Jensen--Zimmermann R--TR2; batching gives a further factor of $3.45$--$23.53$ over sequential calls.
The theoretical results complete the critical dimensions not covered by Gram convexification and quadratic matrix programming.
\end{abstract}
\begin{keywords}
matrix nearness, symplectic Stiefel manifold, semidefinite relaxation, strong duality, Riemannian Hessian
\end{keywords}
\begin{MSCcodes}
15A60, 90C22, 90C26, 90C46
\end{MSCcodes}

\section{Introduction}\label{sec:introduction}
Let $n\ge k\ge1$, write $m=2k$, and define the standard symplectic matrix by
\[
 J_{2r}=\begin{bmatrix}0&I_r\\-I_r&0\end{bmatrix}.
\]
Given $T\in\mathbb R^{2n\times m}$, we consider
\[
 \min\{\|X-T\|_F^2:X\in\mathbb R^{2n\times m},\quad
 X^TJ_{2n}X=J_m\}.
\]
The constraint requires the linear map $X$ to preserve the symplectic form.
The target $T$ is arbitrary and may be rank deficient.
Although the feasible set is noncompact, the squared-distance objective always attains its minimum.
The main questions are when a convex relaxation preserves this minimum and when second-order stationarity implies global optimality.

One motivation for the square problem comes from continuous-variable quantum control.
For quadratic Hamiltonians, the homogeneous action of a quantum gate on the canonical observables is represented by a real symplectic matrix.
Wu--Chakrabarti--Rabitz \cite{WuChakrabartiRabitz2008} measure gate error by the squared Frobenius distance to a prescribed symplectic transformation.
For a symplectic target on the square symplectic group, they prove that the target is the unique local minimizer and all other critical points are saddles \cite{WuChakrabartiRabitz2010}.
Fiori \cite{Fiori2011} recalls this result under the same target assumption and studies means of symplectic matrices.
These landscape results do not cover arbitrary targets in the rectangular problem.

In the complex representation, eliminating a factor with a fixed Gram matrix yields a convex Bures--Wasserstein type objective \cite{BhatiaJainLim2019}.
The rank constraint on the Gram matrix is automatic when $n\ge2k$, so Gram convexification directly gives an exact relaxation in this range.
Beck's quadratic matrix programming theory gives a further sufficient condition for exactness through rank reduction \cite{Beck2007,Beck2009}.
At the critical dimension $n=2k-1$, this condition covers $k=1,2$, but does not directly handle $k\ge3$.

On the rectangular symplectic Stiefel manifold, Son--Absil--Gao--Stykel \cite{SonAbsilGaoStykel2021} prove that every local minimizer of $\operatorname{tr}(X^TMX)$, with $M\succ0$, is global.
Williamson's decomposition $S^TMS=\operatorname{diag}(D,D)$ and the symplectic change of variables $X=SY$ put this model in the zero-target case of our weighted problem, with $H=D$ and $B=0$.
The rank-deficient second-order stationary point criterion in Burer--Monteiro theory \cite{BurerMonteiro2003,BoumalVoroninskiBandeira2020} cannot be applied directly here, because the augmented factor $[U^*;I_n]$ always has full column rank.
Gao and coauthors develop Riemannian optimization methods and Euclidean geometry on the symplectic Stiefel manifold \cite{GaoSonAbsilStykel2021,GaoEuclidean2021}; Jensen and Zimmermann develop trust-region methods using second-order information under the quotient metric induced by a right-invariant metric on the symplectic group \cite{JensenZimmermann2024}. We compare with the former authors' nonmonotone first-order method and the latter authors' R--TR2 variant in section~\ref{sec:numerical}.

We prove that the semidefinite relaxation is exact for every target if and only if $n\ge2k-1$.
Under this condition, the primal problem, the semidefinite relaxation, and both Lagrangian dual problems attain the same optimal value.
For every $k\le n\le2k-2$, we use a rank argument to construct a strict duality gap.
Continuity of the optimal values shows that the gap persists under sufficiently small target perturbations, so the dimension condition cannot be weakened.

For $n\ge2k-1$, we obtain the exact minimum Rayleigh quotient of the Riemannian Hessian at every feasible point under the Euclidean induced metric.
It follows that every second-order stationary point is a global minimizer.
We also derive global objective error bounds in all admissible dimensions, controlling the objective gap by the gradient residual and the negative part of the multiplier matrix $I_m+i\Omega_B(U)$.

We extend the strong-duality result to Hermitian positive definite row weights.
Strong duality holds for every target if and only if the lowest eigenvalue of the weight has complex multiplicity at least $2k-1$.
For the identity weight, this multiplicity is $n$, which recovers the unweighted dimension condition.

For computation, we solve the Lagrangian dual by a Newton method with preconditioned conjugate gradients.
When $n$ is much larger than $k$, the dual formulation reduces the matrix sizes and storage needed for the main iterative state.
It also facilitates batched GPU solves for independent targets.
We compare complete solve times with the Riemannian first-order method and R--TR2 at a common accuracy, as well as sequential and batched calls of the same dual method on the same targets.
Across the six tested sizes, the ratios of mean solve times are $2.55$--$12.60$ for RSD over Dual, $4.61$--$43.93$ for R--TR2 over Dual, and $3.45$--$23.53$ for sequential over batched calls of Dual.

The rest of the paper is organized as follows.
Section~\ref{sec:prelim} relates the primal problem, its semidefinite relaxation, and the two Lagrangian duals.
Section~\ref{sec:curvature} studies the geometry under the Euclidean induced metric and derives the minimum-curvature formula.
Section~\ref{sec:threshold} proves the exact dimension criterion, global optimality of second-order stationary points, and objective error bounds.
Section~\ref{sec:weighted} extends the results to Hermitian positive definite row weights.
Section~\ref{sec:numerical} presents the dual method and numerical experiments.
The final section summarizes the results.
\label{end:introduction}

\Needspace{5\baselineskip}
\section{The primal problem, semidefinite relaxation, and two duals}\label{sec:prelim}\label{sec:convex}
\subsection{Real and complex representations and constraint regularity}
Fix $n\ge k\ge1$, $m=2k$, and $J_{2r}=\left[\begin{smallmatrix}0&I_r\\-I_r&0\end{smallmatrix}\right]$.
Partition the real variable and target by rows as $X=[Q^T,P^T]^T$ and $T=[T_q^T,T_p^T]^T$, and set
$U=Q+iP$ and $B=T_q+iT_p$.
This correspondence preserves the real Frobenius inner product
$\langle V,W\rangle_{\mathbb R}=\operatorname{Re}\operatorname{tr}(V^*W)$ and satisfies
$\operatorname{Im}(U^*U)=X^TJ_{2n}X$. Thus the primal problem
\begin{equation}\label{eq:real-problem}
 \min_{X\in\mathbb R^{2n\times m}}\|X-T\|_F^2,
 \qquad \text{s.t.}\quad X^TJ_{2n}X=J_m
\end{equation}
is equivalent to
\begin{equation}\label{eq:complex-problem}
 F_n(B)=\min_{U\in\mathcal M_{n,k}}f_B(U),\qquad
 f_B(U)=\|U-B\|_F^2,\qquad
 \mathcal M_{n,k}=\{U:\operatorname{Im}(U^*U)=J_m\}.
\end{equation}
Here $U,B\in\mathbb C^{n\times m}$, $\operatorname{Im}$ is taken entrywise, and positive semidefiniteness is understood in the Hermitian sense.
We abbreviate $\mathcal M=\mathcal M_{n,k}$ and write
\[
 \Psi(X)=X^TJ_{2n}X-J_m,\qquad
 \operatorname{skew}(m)=\{S\in\mathbb R^{m\times m}:S^T=-S\}.
\]

\begin{lemma}[Regularity and attainment]\label{lem:regularity}
The feasible set is nonempty and closed and is a smooth real embedded manifold of dimension $2nm-m(m-1)/2$.
The constraint differential at every feasible point is surjective onto $\operatorname{skew}(m)$.
For every $B$, the primal minimum is finite and attained.
\end{lemma}
\begin{proof}
Any $k$ pairs of columns from the standard symplectic basis give a feasible point.
The symplectic constraint implies that $X$ has full column rank.
For any $S\in\operatorname{skew}(m)$, take $H=-XJ_mS/2$. Then
\[
 D\Psi(X)[H]=X^TJ_{2n}H+H^TJ_{2n}X=S.
\]
The regular level-set theorem gives smoothness and the dimension.
Polynomial equalities define a closed feasible set, and the squared-distance objective has bounded nonempty sublevel sets, so the minimum is attained.
\end{proof}

\subsection{The primal Lagrangian dual}
Identify the multiplier space with $\operatorname{skew}(m)$ using the Frobenius inner product
$\langle A,C\rangle_F=\operatorname{tr}(A^TC)$, and define
\begin{equation}\label{eq:lagrangian}
 \mathcal L_B(U,\Omega)=f_B(U)+\langle\Omega,\operatorname{Im}(U^*U)-J_m\rangle_F,
 \qquad K=I_m+i\Omega.
\end{equation}
Expanding gives
\[
 \mathcal L_B(U,\Omega)=\operatorname{tr}(UKU^*)-2\operatorname{Re}\operatorname{tr}(B^*U)
 +\|B\|_F^2-\langle\Omega,J_m\rangle_F.
\]
Define $d_{P,B}(\Omega)=\inf_{U\in\mathbb C^{n\times m}}\mathcal L_B(U,\Omega)$.

\begin{proposition}[Dual function]\label{prop:primal-dual-function}
The value $d_{P,B}(\Omega)$ is finite if and only if
\begin{equation}\label{eq:dual-domain}
 K\succeq0,\qquad B\ker K=\{0\}.
\end{equation}
On this effective domain, set $D=BK^\dagger$. Then
\begin{align}
 d_{P,B}(\Omega)&=\|B\|_F^2-\operatorname{tr}(B^*BK^\dagger)-\langle\Omega,J_m\rangle_F,
 \label{eq:dual-value}\\
 \mathcal L_B(U,\Omega)-d_{P,B}(\Omega)&=\|(U-D)K^{1/2}\|_F^2.
 \label{eq:complete-square}
\end{align}
Elsewhere, $d_{P,B}(\Omega)=-\infty$.
For each fixed $\Omega$ in the effective domain of $d_{P,B}$, the global minimizers of
$\mathcal L_B(\cdot,\Omega)$ over the entire space $\mathbb C^{n\times m}$ are precisely the solutions of $UK=B$.
\end{proposition}
\begin{proof}
If $Kv=\lambda v$ with $\lambda<0$, scaling $U=tzv^*$ for $z\ne0$ makes the quadratic term tend to negative infinity.
If $K\succeq0$ but $Bv\ne0$ for some $v\in\ker K$, choose $U=t(Bv)v^*$; the linear term is $-2t\|Bv\|^2$ and gives the same conclusion.
Conversely, the range condition gives $DK=B$. Completing the square proves both identities, and the squared term vanishes if and only if $(U-D)K=0$.
\end{proof}
Weak duality gives $0=d_{P,B}(0)\le\sup_{\Omega\in\operatorname{skew}(m)}d_{P,B}(\Omega)\le F_n(B)$.

\subsection{Semidefinite relaxation}
Set $G=U^*U$. Then the objective becomes
\[
 \|U-B\|_F^2
 =\operatorname{tr}G-2\operatorname{Re}\operatorname{tr}(B^*U)+\|B\|_F^2,
\]
and the symplectic constraint becomes $\operatorname{Im}G=J_m$.
The primal problem is therefore equivalent to
\[
 \begin{aligned}
 F_n(B)=\min_{U,\;G=G^*}\quad&
 \operatorname{tr}G-2\operatorname{Re}\operatorname{tr}(B^*U)+\|B\|_F^2\\
 \text{s.t.}\quad&\operatorname{Im}G=J_m,\qquad G=U^*U.
 \end{aligned}
\]
Relaxing $G=U^*U$ to $G-U^*U\succeq0$ gives the semidefinite relaxation
\begin{equation}\label{eq:sdp}
 \begin{aligned}
 R_n(B)=\min_{U,\;G=G^*}\quad&
 \rho_B(G,U):=\operatorname{tr}G-2\operatorname{Re}\operatorname{tr}(B^*U)+\|B\|_F^2\\
 \text{s.t.}\quad&\operatorname{Im}G=J_m,\qquad G-U^*U\succeq0.
 \end{aligned}
\end{equation}
By the Schur complement, the last constraint is equivalent to the linear matrix inequality
\[
 \begin{bmatrix}G&U^*\\U&I_n\end{bmatrix}\succeq0.
\]
The choice $U=0$ and $G=tI_m+iJ_m$, with $t>1$, is strictly feasible, so Slater's condition holds.
Moreover, nonempty sublevel sets are closed and bounded, so the relaxation attains its minimum.

\subsection{Pointwise equality of the two dual functions}
For the relaxation, dualize the equality constraint while retaining positive semidefiniteness in the variable domain:
\begin{align}
 \mathcal L_{R,B}(G,U,\Omega)&=\operatorname{tr}(GK)-2\operatorname{Re}\operatorname{tr}(B^*U)
 +\|B\|_F^2-\langle\Omega,J_m\rangle_F,\label{eq:convex-lagrangian}\\
 d_{R,B}(\Omega)&=\inf_{G\succeq U^*U}\mathcal L_{R,B}(G,U,\Omega).
 \label{eq:relax-dual-definition}
\end{align}
\begin{proposition}[The common dual and the relaxation value]\label{prop:dual}
For every $\Omega\in\operatorname{skew}(m)$, $d_{P,B}(\Omega)=d_{R,B}(\Omega)$,
including the case where both values are $-\infty$.
Denote the common function by $d_B$ and set
$D_n(B)=\sup_{\Omega\in\operatorname{skew}(m)}d_B(\Omega)$. For every $n\ge k$ and every $B$,
\begin{equation}\label{eq:relax-dual}
 D_n(B)=R_n(B),
\end{equation}
and an optimal multiplier attaining $D_n(B)$ exists.
\end{proposition}
\begin{proof}
Writing $G=U^*U+S$, $S\succeq0$, gives
\[
 \mathcal L_{R,B}(U^*U+S,U,\Omega)=\mathcal L_B(U,\Omega)+\operatorname{tr}(SK).
\]
If $K\not\succeq0$, Proposition~\ref{prop:primal-dual-function} gives $d_{P,B}(\Omega)=-\infty$.
Fixing $U=0$ and scaling $S$ in a negative eigendirection also gives $d_{R,B}(\Omega)=-\infty$.
If $K\succeq0$, the added term has minimum zero at $S=0$ for each $U$.
The dual functions are therefore equal, including when the range condition fails and the infimum is negative infinity.
Strict feasibility and finiteness of the relaxation allow Slater's theorem \cite[\S5.2.3]{BoydVandenberghe2004} to establish~\eqref{eq:relax-dual} and dual attainment.
\end{proof}

\subsection{Optimality conditions and recovery}
\begin{proposition}[The relaxation--dual gap identity]\label{prop:sdp-equality}
Let $\Omega$ satisfy~\eqref{eq:dual-domain}, and set $D=BK^\dagger$.
For every relaxation feasible point $(G,U)$,
\begin{equation}\label{eq:sdp-gap}
 \rho_B(G,U)-d_B(\Omega)
 =\|(U-D)K^{1/2}\|_F^2+\operatorname{tr}\bigl(K(G-U^*U)\bigr)\ge0.
\end{equation}
The relaxation--dual gap is zero if and only if
\begin{equation}\label{eq:complementarity}
 UK=B,\qquad \operatorname{range}(G-U^*U)\subseteq\ker K.
\end{equation}
Thus a relaxation feasible point and a dual feasible multiplier are both optimal if and only if they satisfy~\eqref{eq:complementarity}.
In particular, $U\in\mathcal M$, $K\succeq0$, and $UK=B$ are sufficient for primal global optimality and zero gap.
\end{proposition}
\begin{proof}
At a relaxation feasible point the Lagrangian equals the objective, so~\eqref{eq:complete-square} gives the identity.
For $S=G-U^*U\succeq0$, we have
$\operatorname{tr}(KS)=\|K^{1/2}S^{1/2}\|_F^2$.
Both terms vanish precisely under~\eqref{eq:complementarity}.
The equivalence with optimality follows from~\eqref{eq:relax-dual} and weak duality.
Finally, $UK=B$ implies $B\ker K=0$; if $U\in\mathcal M$, setting $G=U^*U$ gives the primal sufficient condition.
\end{proof}

\begin{proposition}[Recovery from an optimal multiplier]\label{prop:dual-recovery}
Given any dual optimal multiplier $\Omega_*$, set $K=I_m+i\Omega_*$.
If $K\succ0$, strong duality holds for this target, and the unique primal optimizer is $BK^{-1}$.
If $K$ is singular and strong duality holds, let the columns of $N$ be an orthonormal basis of $\ker K$.
Then the primal optimal set is exactly the nonempty set
\begin{equation}\label{eq:general-kernel-recovery}
 \{D+YN^*:Y\in\mathbb C^{n\times s},\quad D+YN^*\in\mathcal M\},
 \qquad D=BK^\dagger,\quad s=\dim_{\mathbb C}\ker K.
\end{equation}
\end{proposition}
\begin{proof}
If $K\succ0$, complementarity at a relaxation optimizer forces $G=U^*U$.
Thus $U$ is primal feasible and satisfies $UK=B$, or $U=BK^{-1}$.
Every primal optimizer must satisfy the same equation, proving uniqueness.
If $K$ is singular, all solutions of $UK=B$ have the form $D+YN^*$.
Strong duality and primal attainment ensure a symplectically feasible solution;
Proposition~\ref{prop:sdp-equality} identifies these feasible solutions as exactly the primal optimizers.
\end{proof}
\label{end:preliminaries}

\section{Normal projection multipliers and exact minimum curvature}\label{sec:curvature}
Equip $\mathcal M=\mathcal M_{n,k}$ with the Euclidean metric induced by the real Frobenius inner product.
At a fixed feasible point $U$, the minimum curvature is the smallest Hessian Rayleigh quotient over its tangent space:
\begin{equation}\label{eq:min-curvature-definition}
 \lambda_{\min}\bigl(\operatorname{Hess}f_B(U)\bigr)
 =\min_{0\ne Z\in T_U\mathcal M}
 \frac{\langle Z,\operatorname{Hess}f_B(U)[Z]\rangle_{\mathbb R}}
      {\|Z\|_F^2}.
\end{equation}

\subsection{Normal projection, gradient, and Hessian}
The constraint map $\psi(U)=\operatorname{Im}(U^*U)-J_m$ has differential
$D\psi(U)[Z]=\operatorname{Im}(U^*Z+Z^*U)$ and adjoint
$D\psi(U)^*[A]=2iUA$ for $A\in\operatorname{skew}(m)$.
By Lemma~\ref{lem:regularity}, its tangent and normal spaces are
\begin{align}
 T_U\mathcal M&=\{Z:\operatorname{Im}(U^*Z+Z^*U)=0\},\label{eq:tangent}\\
 N_U\mathcal M&=\{iUA:A\in\operatorname{skew}(m)\}.
 \label{eq:normal}
\end{align}

\begin{proposition}[Normal projection multiplier]\label{prop:normal-multiplier}
For any $B\in\mathbb C^{n\times m}$ and $U\in\mathcal M$,
the orthogonal normal projection of $\nabla f_B(U)=2(U-B)$ uniquely determines a real skew-symmetric multiplier $\Omega_B(U)$ through
\begin{equation}\label{eq:normal-gradient-projection}
 P_{N_U\mathcal M}\bigl[\nabla f_B(U)\bigr]=-2iU\Omega_B(U).
\end{equation}
Set $S=\operatorname{Re}(U^*U)=X^TX\succ0$ and $Q_B=U^*B$.
Then $\Omega=\Omega_B(U)$ is the unique real skew-symmetric solution of the Sylvester equation
\begin{equation}\label{eq:sylvester}
 S\Omega+\Omega S=\operatorname{Im}(Q_B+Q_B^*)-2J_m
\end{equation}
and depends smoothly on $U$.
\end{proposition}
\begin{proof}
Surjectivity of $D\psi(U)$ makes its adjoint injective, so the coefficient of the normal projection is unique.
Substituting $2(U-B)+2iU\Omega$ into the tangent condition and using $U^*U=S+iJ_m$ gives~\eqref{eq:sylvester}.
Since $S\succ0$, the spectra of $S$ and $-S$ are disjoint; the Sylvester equation therefore has a unique solution \cite[Chapter~16]{Higham2002}.
Transposition and skew-symmetry of the right-hand side show that the solution is skew-symmetric.
The coefficients depend smoothly on $U$ and the associated operator is everywhere invertible, so the solution also depends smoothly on $U$.
\end{proof}
In the real representation $g_X=2(X-T)-J_{2n}X\Lambda$, we have $\Lambda=2\Omega_B(U)$.

Let $\Pi_U$ denote orthogonal tangent projection, and set $K(U)=I_m+i\Omega_B(U)$.
\begin{proposition}[Gradient and Hessian]\label{prop:hessian}
For any target $B$, feasible point $U$, and $Z\in T_U\mathcal M$,
\begin{equation}\label{eq:gradient}
 g(U)=\operatorname{grad}f_B(U)=2\bigl(UK(U)-B\bigr),
\end{equation}
\begin{equation}\label{eq:hessian}
 \operatorname{Hess}f_B(U)[Z]=\Pi_U(2ZK(U)),\qquad
 \langle Z,\operatorname{Hess}f_B(U)[Z]\rangle_{\mathbb R}
 =2\operatorname{tr}(ZK(U)Z^*).
\end{equation}
\end{proposition}
\begin{proof}
Subtracting the normal projection from the ambient gradient gives~\eqref{eq:gradient}.
Differentiating along a feasible curve yields
\[
 Dg(U)[Z]=2ZK(U)+2iU\,D\Omega_B(U)[Z].
\]
The last term is normal because $D\Omega_B(U)[Z]$ is real skew-symmetric.
The projected-derivative formula for Euclidean embedded manifolds \cite[\S4.9]{EdelmanAriasSmith1998}
gives the Hessian, and $\Pi_U Z=Z$ gives its quadratic form.
\end{proof}

For any Hermitian $K$, with $\lambda=\lambda_{\min}(K)$,
\begin{equation}\label{eq:curvature-ambient-bound}
 \operatorname{tr}(ZKZ^*)-\lambda\|Z\|_F^2
 =\bigl\|Z(K-\lambda I_m)^{1/2}\bigr\|_F^2\ge0,
\end{equation}
with equality if and only if
\begin{equation}\label{eq:curvature-equality-condition}
 Z(K-\lambda I_m)=0\quad\Longleftrightarrow\quad ZK=\lambda Z.
\end{equation}
Thus~\eqref{eq:hessian} gives the minimum-curvature lower bound $2\lambda_{\min}(K(U))$.
We next construct a rank-one direction satisfying both this equality condition and the symplectic tangent condition.

\subsection{Minimum curvature and an attaining direction}
\begin{theorem}[Exact minimum curvature]\label{thm:curvature}
\label{thm:exact-curvature}\label{lem:tangent-rayleigh}
Suppose $n\ge2k-1=m-1$.
For every target $B\in\mathbb C^{n\times m}$ and every feasible point $U$,
the minimum over tangent vectors in~\eqref{eq:min-curvature-definition} is
\begin{equation}\label{eq:exact-curvature}
 \lambda_{\min}\bigl(\operatorname{Hess}f_B(U)\bigr)
 =2\lambda_{\min}(K(U))
 =2\bigl(1-\|\Omega_B(U)\|_2\bigr).
\end{equation}
A tangent vector of complex rank one attains this minimum and is a Hessian eigenvector for the smallest eigenvalue.
More generally, for every real skew-symmetric $\Omega$, with $K=I_m+i\Omega$,
there exists a tangent vector $Z\in T_U\mathcal M$ of complex rank one such that
\begin{equation}\label{eq:rank-one-attainment}
 ZK=\lambda_{\min}(K)Z,\qquad \lambda_{\min}(K)=1-\|\Omega\|_2.
\end{equation}
\end{theorem}
\begin{proof}
Fix $\Omega$ and write $\lambda=\lambda_{\min}(I_m+i\Omega)$ and $\omega=\|\Omega\|_2$.
The purely imaginary eigenvalues of a real skew-symmetric matrix occur in conjugate pairs, so $\lambda=1-\omega$.
To construct a rank-one attaining direction, first choose a corresponding unit eigenvector $v$.
If $\omega>0$, then
\[
 \Omega v=i\omega v,\qquad 0=v^T\Omega v=i\omega v^Tv,
\]
so $v^Tv=0$ and hence
\begin{equation}\label{eq:curvature-eigenvector-coordinates}
 v=\frac{a+ib}{\sqrt2},\qquad
 a,b\in\mathbb R^m,\quad \|a\|=\|b\|=1,\quad a^Tb=0.
\end{equation}
If $\omega=0$, then $K=I_m$; choose any real orthonormal pair $a,b$ and define $v$ by the same formula.

Set $Z=zv^*$, where $z\in\mathbb C^n$ is to be chosen. We show that a nonzero $z$ can be selected so that $Z$ is tangent; this form already gives $ZK=\lambda Z$.
Complete $a,b$ to a real orthogonal matrix $O=[a,b,O_r]$ and write
\[
 UO=[u_1,u_2,E],\qquad ZO=[z,-iz,0]/\sqrt2.
\]
Since $E$ has $m-2$ columns, $\dim_{\mathbb C}\ker E^*\ge n-(m-2)\ge1$.
Choose a nonzero $z_0\in\ker E^*$ and a phase $\theta$ so that $z=e^{i\theta}z_0$ satisfies
\begin{equation}\label{eq:attaining-z}
 E^*z=0,\qquad \operatorname{Re}\bigl((u_1+i u_2)^*z\bigr)=0.
\end{equation}
Such a phase always exists: rotate $(u_1+i u_2)^*z_0$ onto the imaginary axis if it is nonzero, and choose any phase otherwise.

In these coordinates, $E^*z=0$ eliminates the cross blocks, leaving
\[
 \begin{gathered}
 \operatorname{Im}\bigl((UO)^*(ZO)+(ZO)^*(UO)\bigr)
 =\frac1{\sqrt2}
 \begin{bmatrix}0&-c&0\\c&0&0\\0&0&0_{m-2}\end{bmatrix},\\
 c=\operatorname{Re}\bigl((u_1+i u_2)^*z\bigr)=0.
 \end{gathered}
\]
Real orthogonal changes of columns preserve the homogeneous tangent condition, so $Z\in T_U\mathcal M$.
When $m=2$, $E$ is empty and the same argument applies.
Finally, take $\Omega=\Omega_B(U)$. The identity
$\operatorname{Hess}f_B(U)[Z]=\Pi_U(2ZK)=2\lambda Z$
proves the minimum-curvature and eigenvector assertions.
\end{proof}
\label{end:curvature}

\section{The sharp dimension threshold for universal strong duality}\label{sec:threshold}

\begin{theorem}[Dimension criterion for strong duality for every target]\label{thm:threshold}\label{subsec:threshold-statement}
Fix \(n\ge k\ge1\) and let \(m=2k\). The following are equivalent:
\begin{enumerate}
 \item For every \(B\in\mathbb C^{n\times m}\),
 \[
  F_n(B)=R_n(B)=D_n(B);
 \]
 \item \(n\ge m-1\).
\end{enumerate}
Moreover, for every \(k\ge2\) and \(k\le n\le m-2\), there is a nonempty open set of targets,
each with full complex row rank and a real representation of full column rank, such that
\[
 F_n(B)>R_n(B)=D_n(B).
\]
\end{theorem}

\subsection{Proof of the main theorem}
With the notation of section~\ref{sec:prelim}, the primal, relaxation, and dual optimal values satisfy
\(D_n(B)=R_n(B)\le F_n(B)\), and all are attained.

\begin{proof}[Sufficiency]\label{subsec:threshold-sufficiency}
Suppose \(n\ge m-1\). By Lemma~\ref{lem:regularity}, the primal problem for every target \(B\) has a global minimizer
\(U_\star\) satisfying the first- and second-order necessary conditions. Set
\(\Omega_\star=\Omega_B(U_\star)\) and \(K_\star=I_m+i\Omega_\star\). The first-order condition and Theorem~\ref{thm:curvature} give
\[
 U_\star K_\star=B,\qquad
 0\le\lambda_{\min}\bigl(\operatorname{Hess}f_B(U_\star)\bigr)
   =2\lambda_{\min}(K_\star).
\]
Thus \(K_\star\succeq0\), and the first-order condition implies \(B\ker K_\star=\{0\}\).
The completed-square identity~\eqref{eq:complete-square} attains equality at \(U_\star\), so
\[
 F_n(B)=d_B(\Omega_\star)\le D_n(B)=R_n(B)\le F_n(B).
\]
This proves sufficiency, including rank-deficient targets and the endpoint \(k=n=1\).
\end{proof}

\medskip\noindent\textbf{Necessity: a two-row, four-column counterexample.}\label{subsec:base-gap}
\par\nopagebreak
Proposition~\ref{prop:dual-recovery} shows that a positive definite optimal $K$ guarantees exactness,
so a strict gap can occur only on the singular dual boundary.
For the counterexamples below, order columns in symplectic pairs so that $J_m=\operatorname{diag}(J_2,\ldots,J_2)$.
The inverse column permutation preserves norms, ranks, and optimal values and restores the standard column order.
Take
\begin{equation}\label{eq:gap-base-data}
 B_0=\begin{bmatrix}2&2i&1&i\\4i&-4&0&1\end{bmatrix},
 \qquad U_0=\begin{bmatrix}2&0&1&i\\2i&-2&0&1\end{bmatrix},
\end{equation}
\begin{equation}\label{eq:gap-base-multiplier}
 \Omega_0=\operatorname{diag}(J_2,0_2),\qquad K_0=I_4+i\Omega_0.
\end{equation}

\begin{lemma}[A strict gap with two rows and four columns]\label{lem:base-gap}
For the target in~\eqref{eq:gap-base-data},
\[
 R_2(B_0)=D_2(B_0)=18<F_2(B_0).
\]
\end{lemma}
\begin{proof}
Set
\[
 S_0=3\begin{bmatrix}1&-i&0&0\\i&1&0&0\\0&0&0&0\\0&0&0&0\end{bmatrix},
 \qquad G_0=U_0^*U_0+S_0.
\]
Direct calculation gives $K_0\succeq0$, $B_0\ker K_0=0$, $S_0\succeq0$, and
\begin{equation}\label{eq:gap-base-values}
 \begin{gathered}
 \operatorname{Im}G_0=J_4,\qquad U_0K_0=B_0,\qquad K_0S_0=0,\\
 d_{B_0}(\Omega_0)=43-23-2=18.
 \end{gathered}
\end{equation}
Thus $(G_0,U_0)$ is relaxation feasible and satisfies complementarity~\eqref{eq:complementarity}.
The gap identity gives $R_2(B_0)=D_2(B_0)=18$.

Let $U_\star$ be a primal global minimizer.
If its objective value were also $18$, Proposition~\ref{prop:sdp-equality} would require $U_\star K_0=B_0$, or
\begin{equation}\label{eq:gap-equality-family}
 \begin{gathered}
 U_\star=[d_0+h,i(d_0-h),a,b],\qquad h\in\mathbb C^2,\\
 d_0=\begin{bmatrix}1\\2i\end{bmatrix},\quad
 a=\begin{bmatrix}1\\0\end{bmatrix},\quad
 b=\begin{bmatrix}i\\1\end{bmatrix}.
 \end{gathered}
\end{equation}
Vanishing of the imaginary inner products of the first two columns with $a$ forces $h_1=1$;
comparison with $b$ then gives $h_2=0$.
But the first symplectic pair satisfies
\[
 \operatorname{Im}\bigl((d_0+h)^*i(d_0-h)\bigr)
 =\|d_0\|^2-\|h\|^2=4,
\]
contradicting the required value $1$. Hence $F_2(B_0)>18$.
\end{proof}

\medskip\noindent\textbf{Embedding into every dimension below the threshold.}\label{subsec:gap-embedding}
\par\nopagebreak
\begin{proposition}[Strict gaps in every lower dimension]\label{prop:all-size-gap}
Let $k\ge2$ and $k\le n\le2k-2$, and set
\[
 r=n-k\ge0,\qquad s=2k-n-2\ge0,\qquad r+s=k-2.
\]
Define
\begin{equation}\label{eq:gap-embedding-blocks}
 C_1=[1,i],\qquad
 C_2=\begin{bmatrix}\sqrt2&i/\sqrt2\\0&\sqrt{3/2}\end{bmatrix},
\end{equation}
and, in symplectic-pair column order, let
\begin{equation}\label{eq:gap-general-target}
 B=\operatorname{diag}
 \bigl(B_0,\underbrace{C_1,\ldots,C_1}_{s},
            \underbrace{C_2,\ldots,C_2}_{r}\bigr).
\end{equation}
Then $B\in\mathbb C^{n\times2k}$ has full complex row rank, its real representation has full column rank, and
\[
 F_n(B)>R_n(B)=D_n(B)=18.
\]
\end{proposition}
\begin{proof}
The added blocks satisfy $C_1^*C_1=I_2+iJ_2$ and $C_2^*C_2=2I_2+iJ_2$ and are therefore symplectically feasible.
The row and column counts are $2+s+2r=n$ and $4+2s+2r=m$; zero repetitions mean empty blocks.
Take
\[
 \Omega=\operatorname{diag}(\Omega_0,0_{m-4}),\qquad
 K=\operatorname{diag}(K_0,I_{m-4}).
\]
The norm and pseudoinverse terms of each added block cancel in the dual objective, so $d_B(\Omega)=18$.
Combining $(G_0,U_0)$ and the pairs $(C_j^*C_j,C_j)$ block diagonally gives a relaxation feasible point of value $18$.
Thus $D_n(B)=R_n(B)=18$.

Let $E$ be the last $m-2$ columns of $B$. Then
\[
 E=\operatorname{diag}\left(
 \begin{bmatrix}1&i\\0&1\end{bmatrix},
 \underbrace{C_1,\ldots,C_1}_{s},\underbrace{C_2,\ldots,C_2}_{r}\right)
\]
has full complex row rank.
By Lemma~\ref{lem:regularity}, a primal global minimizer $U_\star$ exists.
Suppose, for contradiction, that $f_B(U_\star)=18$.
The gap identity gives $U_\star K=B$, so
\begin{equation}\label{eq:gap-embedded-family}
 \begin{gathered}
 U_\star=[d_e+h,i(d_e-h),E],\qquad h\in\mathbb C^n,\\
 d_e=(1,2i,0,\ldots,0)^T.
 \end{gathered}
\end{equation}
To determine $h$, take $h_0=e_1$; the first two columns in this form then satisfy all cross constraints with $E$.
Comparing the cross constraints for $h$ and $h_0$, and writing $\zeta=h-h_0$, gives, for every column $f$ of $E$,
\[
 \operatorname{Im}(\zeta^*f)=0,\qquad
 \operatorname{Im}(i\zeta^*f)=\operatorname{Re}(\zeta^*f)=0.
\]
Hence $E^*\zeta=0$, and full row rank implies $h=h_0$.
The imaginary inner product of the first symplectic pair is then $\|d_e\|^2-\|h_0\|^2=4$,
contradicting the required value $1$. This proves $F_n(B)>18$.
The ranks of $B_0$ and the added blocks, together with their disjoint row supports, also give both asserted full-rank properties.
\end{proof}

This proves necessity in Theorem~\ref{thm:threshold}.
Factoring the slack as $G-U^*U=V^*V$ and adding $m$ rows, or conversely splitting off these rows, gives
\begin{equation}\label{eq:extra-rows}
 R_n(B)=F_{n+m}\left(\begin{bmatrix}B\\0_{m\times m}\end{bmatrix}\right).
\end{equation}
This augmented representation and continuity of the distance function show that both optimal values depend continuously on the target.
The strict gap therefore persists under sufficiently small perturbations, as do the full-rank properties, which are open conditions.
These counterexamples also show that the dimension condition in the exact minimum-curvature formula cannot be weakened.

\subsection{Globality of second-order stationary points and objective error bounds}\label{sec:landscape}
A feasible point is second-order stationary if \(\operatorname{grad}f_B(U)=0\) and \(\operatorname{Hess}f_B(U)\succeq0\).

\begin{corollary}[Globality of second-order stationary points]\label{cor:sosp}\label{thm:sosp-global}
If \(n\ge2k-1\), then for every target \(B\), all second-order stationary points are global minimizers.
Every nonglobal first-order stationary point admits a strictly negative curvature direction given by the rank-one construction in Theorem~\ref{thm:curvature}.
\end{corollary}
\begin{proof}
At a second-order stationary point, stationarity gives \(UK=B\), and Theorem~\ref{thm:curvature} gives \(K\succeq0\).
By~\eqref{eq:complete-square}, every feasible \(V\) satisfies
\begin{equation}\label{eq:global-completion}
 f_B(V)-f_B(U)
 =\operatorname{tr}\bigl((V-U)K(V-U)^*\bigr)\ge0.
\end{equation}
At a nonglobal stationary point, \(K\) cannot be positive semidefinite, so the exact curvature formula and its attaining direction give strictly negative curvature.
\end{proof}

\begin{proposition}[Inverse-free objective error bounds in all dimensions]\label{prop:inverse-free}\label{thm:inverse-free}
Let \(n\ge k\), let \(B\) be arbitrary, and let \(U\) be exactly feasible. Set
\(K(U)=I_m+i\Omega_B(U)\) and write
\[
 F=f_B(U),\quad F_*=F_n(B),\quad
 \gamma=\|\operatorname{grad}f_B(U)\|_F,\quad
 \delta=\max\{0,-\lambda_{\min}(K(U))\}.
\]
Then
\begin{align}
 F_*&\ge
 \left[\max\left\{0,
 \frac{(1-\delta)\sqrt F-\gamma}{1+\delta}\right\}\right]^2,
 \label{eq:inverse-free-sharp}\\
 0\le F-F_*&\le2\sqrt F\,\gamma+4\delta F.
 \label{eq:inverse-free-coarse}
\end{align}
Neither dual feasibility nor primal strong duality is required.
\end{proposition}
\begin{proof}
Write \(g(U)=\operatorname{grad}f_B(U)\). Expanding the Lagrangian with the multiplier at \(U\) fixed gives, for every feasible \(V\),
\begin{equation}\label{eq:nonstationary-completion}
 f_B(V)-f_B(U)
 =\langle g(U),V-U\rangle_{\mathbb R}
  +\operatorname{tr}\bigl((V-U)K(U)(V-U)^*\bigr).
\end{equation}
Choose a primal minimizer \(V_*\) and set \(\rho=\|V_*-U\|_F\). Cauchy--Schwarz and
\(K(U)\succeq-\delta I_m\) yield
\begin{equation}\label{eq:inverse-free-distance-step}
 0\le F-F_*\le\gamma\rho+\delta\rho^2.
\end{equation}
Let \(a=\sqrt F\) and \(b=\sqrt{F_*}\). Then \(0\le b\le a\), and the triangle inequality gives
\(\rho\le a+b\). Hence
\[
 (a-b)(a+b)\le\gamma(a+b)+\delta(a+b)^2.
\]
If \(a+b=0\), the result holds. Otherwise, division by \(a+b\) gives
\((1+\delta)b\ge(1-\delta)a-\gamma\), which, together with \(b\ge0\), proves~\eqref{eq:inverse-free-sharp}.
Using \(\rho\le2a\) in~\eqref{eq:inverse-free-distance-step} proves~\eqref{eq:inverse-free-coarse}.
\end{proof}

\begin{corollary}[Global objective error at approximate second-order stationary points]\label{cor:approx-sosp}
Let \(n\ge2k-1\), let \(B\) be arbitrary, and let \(U\) be exactly feasible with
\[
 \|\operatorname{grad}f_B(U)\|_F\le\varepsilon_g,\qquad
 \operatorname{Hess}f_B(U)\succeq-\varepsilon_H I,
 \qquad \varepsilon_g,\varepsilon_H\ge0,
\]
where the Hessian inequality is on the real tangent space. Then
\begin{equation}\label{eq:sosp-global-error}
 0\le f_B(U)-F_*\le
 2\sqrt{f_B(U)}\,\varepsilon_g+2 f_B(U)\varepsilon_H,
 \qquad F_*=F_n(B).
\end{equation}
\end{corollary}
\begin{proof}
Theorem~\ref{thm:curvature} gives \(\delta\le\varepsilon_H/2\). Combining this with \(\gamma\le\varepsilon_g\)
in~\eqref{eq:inverse-free-coarse} proves the result.
\end{proof}

These bounds control suboptimality through residuals and negative eigenvalues; see also~\cite[Lemma~6]{BoumalVoroninskiBandeira2016}.
They require no inverse of $K$ and therefore allow singular multipliers.

\label{end:threshold}

\section{Strong duality under Hermitian row weighting}\label{sec:weighted}

\subsection{The weighted problem and main theorem}\label{subsec:weighted-main}
We extend the squared-distance problem to Hermitian positive definite row weights; $H=I_n$ recovers the preceding problem.
Fix \(n\ge k\ge1\), \(m=2k\), and a Hermitian positive definite matrix
\(H\in\mathbb C^{n\times n}\). Consider
\begin{equation}\label{eq:weighted-objective}
 F_H(B)=\min_{U\in\mathcal M_{n,k}}f_{H,B}(U),\qquad
 f_{H,B}(U)=\operatorname{tr}\bigl((U-B)^*H(U-B)\bigr).
\end{equation}
The target \(B\in\mathbb C^{n\times m}\) is arbitrary, the constraint remains
\(\operatorname{Im}(U^*U)=J_m\), and the induced Euclidean metric is unchanged.
The weight \(H\) enters only the objective. If \(H=H_R+iH_I\), its real representation is
\[
 \begin{bmatrix}H_R&-H_I\\H_I&H_R\end{bmatrix}.
\]
This matrix commutes with the standard complex structure. The Euclidean curvature analysis below assumes this structure.

For a real skew-symmetric multiplier \(\Omega\), define
\begin{equation}\label{eq:weighted-lagrangian}
 \begin{aligned}
 \mathcal L_{H,B}(U,\Omega)
 &=f_{H,B}(U)+\langle\Omega,\operatorname{Im}(U^*U)-J_m\rangle_F,\\
 q_{H,B}(\Omega)&=\inf_{U\in\mathbb C^{n\times m}}\mathcal L_{H,B}(U,\Omega).
 \end{aligned}
\end{equation}
Write \(D_H(B)=\sup_{\Omega^T=-\Omega}q_{H,B}(\Omega)\) for the standard Lagrangian dual value, and let
\begin{equation}\label{eq:weighted-min-eigenspace}
 \alpha=\lambda_{\min}(H)>0,\qquad
 \mathcal E_\alpha=\ker(H-\alpha I_n),\qquad
 d=\dim_{\mathbb C}\mathcal E_\alpha.
\end{equation}

\begin{theorem}[Smallest-eigenvalue multiplicity and universal strong duality]\label{thm:weighted}
For a fixed \(H=H^*\succ0\), standard Lagrangian strong duality
\(F_H(B)=D_H(B)\) holds for every \(B\in\mathbb C^{n\times m}\)
if and only if \(d\ge m-1\).
\end{theorem}

In the unweighted case, $d=n$, so the criterion agrees with Theorem~\ref{thm:threshold}.

\subsection{The weighted dual and attainment}\label{subsec:weighted-dual}

\begin{proposition}[Dual representation, attainment, and equality conditions]\label{prop:weighted-dual}
For every target, the primal problem and its standard Lagrangian dual have finite attained optimal values.
In unitary row coordinates, let \(H=\operatorname{diag}(a_1,\ldots,a_n)\), where
\(a_1=\cdots=a_d=\alpha\) and all remaining eigenvalues exceed \(\alpha\).
Let \(u_j,b_j\) denote the rows of \(U,B\) indexed by \(j\), and set \(K_j=a_jI_m+i\Omega\).
The effective domain of the dual function is
\begin{equation}\label{eq:weighted-dual-domain}
 \mathcal D_{H,B}
 =\{\Omega:\Omega^T=-\Omega,\quad
        K_j\succeq0,\quad b_j\ker K_j=\{0\}\quad(1\le j\le n)\}.
\end{equation}
On this domain,
\begin{equation}\label{eq:weighted-dual}
 q_{H,B}(\Omega)=\sum_{j=1}^n
 \left(a_j\|b_j\|^2-a_j^2b_jK_j^\dagger b_j^*\right)
     -\langle\Omega,J_m\rangle_F;
\end{equation}
elsewhere \(q_{H,B}(\Omega)=-\infty\).
All \(K_j\) are positive semidefinite if and only if \(\|\Omega\|_2\le\alpha\);
the effective domain also requires the kernel conditions above.

For any \(\Omega\in\mathcal D_{H,B}\), let \(c_j=a_jb_jK_j^\dagger\). Then
\begin{equation}\label{eq:weighted-square}
 \mathcal L_{H,B}(U,\Omega)-q_{H,B}(\Omega)
 =\sum_{j=1}^n\|(u_j-c_j)K_j^{1/2}\|^2.
\end{equation}
Hence a feasible point attains the dual value at this multiplier if and only if
\begin{equation}\label{eq:weighted-equality}
 u_jK_j=a_jb_j\qquad(1\le j\le n).
\end{equation}
\end{proposition}

\begin{proof}
A unitary row transformation preserves the constraint, objective, and Euclidean inner product.
In diagonal coordinates, the quadratic and linear terms for row $j$ are $u_jK_ju_j^*$ and $-2a_j\operatorname{Re}(b_ju_j^*)$.
Applying the square-completion argument of Proposition~\ref{prop:primal-dual-function} to each row gives the effective domain, dual formula, and equality conditions.
Since $\lambda_{\min}(i\Omega)=-\|\Omega\|_2$, all $K_j$ are positive semidefinite precisely when $\|\Omega\|_2\le\alpha$.

The primal feasible set is nonempty and closed, and $f_{H,B}(U)\ge\alpha\|U-B\|_F^2$, so the primal minimum is attained.
The dual function is an infimum of continuous affine functions and hence is upper semicontinuous.
It is $-\infty$ outside the compact ball $\{\Omega\in\operatorname{skew}(m):\|\Omega\|_2\le\alpha\}$,
equals zero at the origin, and has a finite upper bound by weak duality.
Its maximum is therefore also attained.
\end{proof}

\subsection{Minimum curvature of the weighted problem}\label{subsec:weighted-curvature}
Following the normal projection construction in Section~\ref{sec:curvature}, define the unique real skew-symmetric multiplier
$\Omega_{H,B}(U)$ at a feasible point $U$ by
\begin{equation}\label{eq:weighted-normal-projection}
 P_{N_U\mathcal M}\bigl(2H(U-B)\bigr)=-2iU\Omega_{H,B}(U).
\end{equation}
The metric is unchanged, so uniqueness and smooth dependence follow by the argument of Proposition~\ref{prop:normal-multiplier}.

\begin{proposition}[Weighted minimum curvature]\label{prop:weighted-curvature}
Suppose the smallest eigenvalue $\alpha$ of $H=H^*\succ0$ has complex multiplicity $d\ge m-1$.
For every target $B\in\mathbb C^{n\times m}$ and feasible point $U\in\mathcal M_{n,k}$,
the minimum curvature under the induced Euclidean metric is
\begin{equation}\label{eq:weighted-curvature}
 \lambda_{\min}\bigl(\operatorname{Hess}f_{H,B}(U)\bigr)
 =2\bigl(\alpha-\|\Omega_{H,B}(U)\|_2\bigr),
\end{equation}
and is attained by a tangent vector of complex rank one.
\end{proposition}
\begin{proof}
Write $\Omega=\Omega_{H,B}(U)$ and $\omega=\|\Omega\|_2$.
The projected-derivative calculation in Proposition~\ref{prop:hessian} gives, for $Z\in T_U\mathcal M$,
\begin{equation}\label{eq:weighted-rayleigh-gap}
 \begin{aligned}
 \tfrac12\langle Z,\operatorname{Hess}f_{H,B}(U)[Z]\rangle_{\mathbb R}
 &=\operatorname{tr}(Z^*HZ)+\operatorname{tr}(Zi\Omega Z^*)\\
 &\ge(\alpha-\omega)\|Z\|_F^2.
 \end{aligned}
\end{equation}
Equality holds if and only if $HZ=\alpha Z$ and $Zi\Omega=-\omega Z$.
Choose the unit vector $v=(a+ib)/\sqrt2$ and real orthogonal matrix $O=[a,b,O_r]$ as in the proof of Theorem~\ref{thm:curvature},
and write $UO=[u_1,u_2,E]$.
For $Z=zv^*$ to satisfy both equality conditions, $z$ must lie in the smallest eigenspace $\mathcal E_\alpha$.
Since $E$ has $m-2$ columns,
\[
 \dim_{\mathbb C}(\mathcal E_\alpha\cap\ker E^*)\ge d-(m-2)\ge1.
\]
Choose a nonzero vector in this intersection and adjust its phase as in Theorem~\ref{thm:curvature} so that
$\operatorname{Re}((u_1+i u_2)^*z)=0$.
Then $Z=zv^*$ is a tangent vector of complex rank one attaining the lower bound.
The construction includes $\omega=0$ and $m=2$.
\end{proof}

\subsection{Proof of the multiplicity criterion}\label{subsec:weighted-proof}

\begin{proof}[Proof of Theorem~\ref{thm:weighted}]
\emph{Sufficiency.}
Suppose $d\ge m-1$. At a primal global minimizer $U_*$, the first- and second-order necessary conditions and Proposition~\ref{prop:weighted-curvature} give
\[
 H(U_*-B)+iU_*\Omega_*=0,\qquad \|\Omega_*\|_2\le\alpha.
\]
Consequently,
\[
 K_j=a_jI_m+i\Omega_*
 \succeq(a_j-\|\Omega_*\|_2)I_m
 \succeq(a_j-\alpha)I_m\succeq0,
\]
and the equations $u_jK_j=a_jb_j$ ensure the range conditions.
Equation~\eqref{eq:weighted-square} gives $f_{H,B}(U_*)=q_{H,B}(\Omega_*)$, establishing strong duality.

\emph{Necessity.}
Work in diagonal row coordinates and choose targets with \(b_j=0\) for \(j>d\).
For every multiplier with finite dual value,
\[
 K_j=(a_j-\alpha)I_m+(\alpha I_m+i\Omega)\succ0,\qquad j>d.
\]
If a primal feasible point attains that dual value, \eqref{eq:weighted-equality} forces \(u_j=0\) in all rows with higher weights.

First suppose \(d<k\), and take \(B=0\).
The effective dual domain is exactly \(\|\Omega\|_2\le\alpha\), and
\(q_{H,0}(\Omega)=-\langle\Omega,J_m\rangle_F\).
Since every singular value of \(J_m\) is \(1\), this expression is at most \(\alpha m\),
attained at \(\Omega=-\alpha J_m\).
If a primal feasible point had the same value, vanishing of the rows with higher weights would imply
\(\operatorname{rank}_{\mathbb C}U\le d<k\), contradicting
\[
 2k=\operatorname{rank}_{\mathbb R}\operatorname{Im}(U^*U)
 \le2\operatorname{rank}_{\mathbb C}(U^*U)
 =2\operatorname{rank}_{\mathbb C}U
\]
Primal attainment therefore gives \(F_H(0)>\alpha m=D_H(0)\).

Now suppose \(k\le d\le m-2\).
Take the \(d\)-by-\(m\) target \(B_{\rm low}\) constructed in the proof of Theorem~\ref{thm:threshold};
its primal value exceeds \(18\), while its dual value is \(18\).
Denote an optimal multiplier by \(\Omega_0\in\operatorname{skew}(m)\), and set
\[
 B=\begin{bmatrix}B_{\rm low}\\0\end{bmatrix},\qquad
 \Omega=\alpha\widetilde\Omega,\qquad
 \widetilde K=I_m+i\widetilde\Omega.
\]
For the first \(d\) rows, \(K_j=\alpha\widetilde K\) and
\(K_j^\dagger=\alpha^{-1}\widetilde K^\dagger\).
Their contribution together with the constraint term is therefore \(\alpha\) times the lower-dimensional dual function.
All remaining target rows are zero and contribute zero.
On the lower-dimensional effective domain, their coefficients
\((a_j-\alpha)I_m+\alpha\widetilde K\) are positive definite and impose no additional kernel condition.
Conversely, failure of a lower-dimensional semidefinite or kernel condition makes the weighted dual value \(-\infty\) as well.
Consequently, the following extended-real identity holds on the entire multiplier space:
\begin{equation}\label{eq:weighted-dual-scaling}
 q_{H,B}(\alpha\widetilde\Omega)
   =\alpha d_{B_{\rm low}}(\widetilde\Omega),
 \qquad \widetilde\Omega^T=-\widetilde\Omega.
\end{equation}
It follows that \(D_H(B)=18\alpha\), attained at \(\alpha\Omega_0\).
A primal point with this value would have all rows with higher weights equal to zero; the remaining \(d\) rows would form a feasible point of value \(18\) for the lower-dimensional problem, a contradiction.
Primal attainment then gives \(F_H(B)>18\alpha=D_H(B)\).
The two cases exhaust \(d<m-1\).
\end{proof}

\subsection{Second-order stationary points and anisotropic perturbations}\label{subsec:weighted-consequences}

\begin{corollary}[Global optimality of weighted second-order stationary points]\label{cor:weighted-sosp}
If \(d\ge m-1\), then for every target \(B\), every second-order stationary point is globally minimizing.
Every nonglobal first-order stationary point has a direction of strictly negative curvature.
\end{corollary}

\begin{proof}
The second-order condition and Proposition~\ref{prop:weighted-curvature} give \(\|\Omega\|_2\le\alpha\).
Together with stationarity, the square-completion identity gives, for every feasible \(V\),
\begin{equation}\label{eq:weighted-difference}
 f_{H,B}(V)-f_{H,B}(U)
 =\operatorname{tr}\bigl((V-U)^*H(V-U)\bigr)
   +\operatorname{tr}\bigl((V-U)i\Omega(V-U)^*\bigr)\ge0.
\end{equation}
A nonglobal first-order stationary point must therefore satisfy \(\|\Omega\|_2>\alpha\);
the exact curvature formula and its attaining direction give strictly negative curvature.
\end{proof}

The next example shows that arbitrarily small perturbations of the identity weight can destroy the property that strong duality holds for every target.
\begin{example}[Arbitrarily small anisotropic perturbations]\label{ex:anisotropy}
Take \(n=3,k=2\), order the columns in symplectic pairs, and use \(B_0\) from \eqref{eq:gap-base-data}. Set
\[
 H_\varepsilon=\operatorname{diag}(1,1,1+\varepsilon),\qquad
 B=\begin{bmatrix}B_0\\0\end{bmatrix},\qquad \varepsilon>0.
\]
Here \(\alpha=1\) and \(d=2<m-1=3\).
By \eqref{eq:weighted-dual-scaling} and the necessity proof above, the dual value is \(18\) and the primal value is strictly larger.
Take the three-row feasible matrix
\begin{equation}\label{eq:gap-base-lift}
 \widehat U=\begin{bmatrix}
 2&0&1&i\\2i&-2&0&1\\\sqrt3&-i\sqrt3&0&0
 \end{bmatrix}.
\end{equation}
Its unweighted objective value is $18$, and its third row has squared norm $6$. Hence
\begin{equation}\label{eq:anisotropy-gap}
 18<\min_{U\in\mathcal M_{3,2}}f_{H_\varepsilon,B}(U)
       \le18+6\varepsilon.
\end{equation}
Since \(\|H_\varepsilon-I\|_2=\varepsilon\), arbitrarily small anisotropic perturbations can destroy the property that strong duality holds for every target at the unweighted critical dimension.
\end{example}

The necessity in Theorem~\ref{thm:weighted} concerns strong duality for all targets, not the global optimality of local minima.
For example, every local minimum of the symplectic trace problem \(B=0\) is global
\cite[Theorem~4.6, Proposition~4.11]{SonAbsilGaoStykel2021},
yet a strict duality gap occurs when \(d<k\).

\label{end:weighted}

\section{Numerical experiments}\label{sec:numerical}
The dual formulation offers a way to reduce the size of the matrices updated during optimization.
Primal methods iterate on a $2n\times2k$ matrix, whereas the multiplier $\Omega$ is a $2k\times2k$ skew-symmetric matrix.
After the primal variable is eliminated, the target enters the dual objective through the Gram matrix $B^*B$ of the same order; see~\eqref{eq:dual-value}.
Thus, when $n$ is large relative to $k$, the main iterative state requires $O(k^2)$ storage rather than $O(nk)$.
The full target is still used to form the Gram matrix and recover the primal solution.

The smaller iterative state also makes it practical to solve many independent targets together.
Targets of the same size have separate Gram matrices and multipliers, whose updates can be carried out concurrently using batched matrix operations on the GPU.
We examine these computational benefits at dimensions for which strong duality holds.
The comparison with two Riemannian methods measures the complete cost of a single-target solve, while the comparison between sequential and batched dual solves measures the additional benefit of parallelism across targets.

\subsection{Experimental setup}\label{subsec:numerical-methods}\label{subsec:numerical-validation}
The dual method, denoted Dual, maximizes $d_B$ by minimizing the equivalent real objective
\[
 h_0(\Omega)=\operatorname{tr}(C_rP-C_iQ)+\langle\Omega,J_m\rangle_F
             =\|B\|_F^2-d_B(\Omega),
\]
where $C_r+iC_i=B^*B$, $P=(I_m+\Omega^2)^{-1}$, and $Q=-\Omega P$.
We use Newton's method with preconditioned conjugate gradients, starting from $\Omega=0$.
Backtracking keeps $I_m+i\Omega\succ0$ and decreases $h_0$; the primal variable is recovered as $U=B(P+iQ)$.
The tolerance is $5\times10^{-9}$, with at most $100$ iterations.

The first-order baseline, denoted RSD, is Gao's Euclidean-metric method with Cayley retraction and nonmonotone line search \cite{GaoSonAbsilStykel2021,GaoEuclidean2021,GaoSpoptCode}; our port retains its step and line-search rules.
We also compare with the trust-region method of Jensen and Zimmermann \cite{JensenZimmermann2024,JensenZimmermannCode2024}, using the quotient metric induced by a right-invariant metric on the symplectic group and their retraction formed from two Cayley factors.
Their R--TR1 variant uses the exact Riemannian Hessian, whereas R--TR2 projects the ordinary directional derivative of the Riemannian gradient onto the tangent space.
We use R--TR2, the default in the authors' public code, with its original Hessian approximation and trust-region updates.
We ported the original MATLAB/Manopt implementation \cite{BoumalMishraAbsilSepulchre2014} to Python and made the necessary computational optimizations.

Both Riemannian methods start from $U_0=[Q_0,iQ_0]$, where $Q_0$ is an orthonormal polar factor of $B_1-iB_2$ for $B=[B_1,B_2]$.
They use the common residual tolerances below, with at most $1000$ iterations or $60$ seconds per target.
All three methods use the same GPU and double precision.
Solve times include initialization, data transfers, and recovery of the output, but exclude target generation and the independent accuracy checks.

Parallel computation uses an RTX 3060 with $6$ GB of GPU memory; the CPU is an Intel Core i7-12700H.\footnote{The code is publicly available at \url{https://github.com/yulinxie20/symplectic-stiefel-nearness}.}

We use all six combinations of $n\in\{2048,4096\}$ and $k\in\{16,32,48\}$.
For each size, we generate $64$ targets with independent $\mathcal N(0,1/(2n))$ entries.
We compare the single-target solve times of Dual, RSD, and R--TR2, and the total times for sequential and batched solves with Dual on the same targets.
Each experiment is repeated three times, and running times are reported as arithmetic means.
Speedups are ratios of the corresponding mean times.

For a computed $X$, let $\Omega_{\rm p}$ be its projection multiplier and evaluate
\begin{equation}\label{eq:numerical-residuals}
 r_{\rm F}=\frac{\|X^TJ_{2n}X-J_m\|_F}{\sqrt m},\qquad
 r_{\rm G}=\frac{\|X-T-J_{2n}X\Omega_{\rm p}\|_F}{\max\{1,\|T\|_F\}}.
\end{equation}
Here $r_{\rm F}$ checks the symplectic constraint, and $r_{\rm G}$ checks the first-order optimality condition.
We also compute
\begin{equation}\label{eq:numerical-gap}
 \delta=\frac{|f(X)-d_B(\Omega_{\rm p})|}{\max\{1,|f(X)|\}},
\end{equation}
which measures the difference between the objective value and the dual value.
A solve is successful if it terminates normally and satisfies $r_{\rm F}\le10^{-8}$, $r_{\rm G}\le10^{-9}$, and $\delta\le10^{-8}$.

\subsection{Experimental results}\label{subsec:numerical-random}\label{subsec:numerical-batch}
Figure~\ref{fig:numerical-comparison} summarizes both comparisons and shows variation across the three rounds; exact values are listed in Tables~\ref{tab:numerical-time} and~\ref{tab:numerical-batch}.

\begin{figure}[htbp]
\centering
\includegraphics[width=\textwidth]{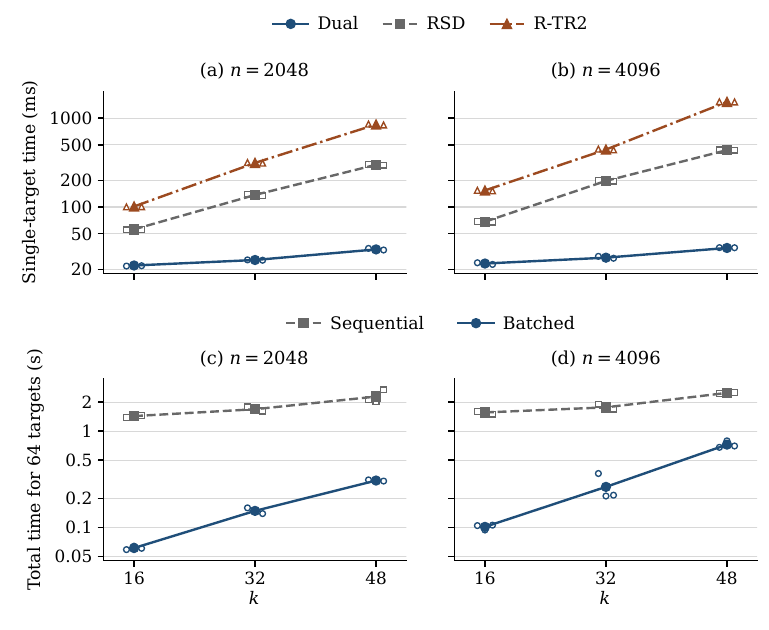}
\caption{Single-target times (top, milliseconds) and total times for $64$ targets (bottom, seconds). Large markers and lines show three-round arithmetic means. Small open markers show each round's values (means over $64$ targets in the top row) and are displaced horizontally for visibility. Time axes are logarithmic.}\label{fig:numerical-comparison}
\end{figure}
\FloatBarrier

Table~\ref{tab:numerical-time} lists the mean single-target solve times of Dual, RSD, and R--TR2.

\begin{table}[htbp]
\centering\small
\setlength{\tabcolsep}{3.5pt}
\caption{Mean single-target times in milliseconds. Success counts are listed in the order Dual/RSD/\mbox{R--TR2}, each out of $192$.}\label{tab:numerical-time}
\begin{tabular}{rrrrrrrr}
\toprule
$n$ & $k$ & Dual & RSD & R--TR2 & RSD/Dual & R--TR2/Dual & Successes\\
\midrule
2048 & 16 & 22.0 & 56.1 & 101.5 & 2.55 & 4.61 & 192/192/192 \\
2048 & 32 & 25.4 & 136.8 & 313.7 & 5.39 & 12.36 & 192/192/192 \\
2048 & 48 & 33.4 & 298.9 & 847.8 & 8.95 & 25.39 & 192/192/192 \\
4096 & 16 & 23.2 & 68.4 & 153.7 & 2.95 & 6.64 & 192/192/192 \\
4096 & 32 & 27.0 & 196.7 & 443.1 & 7.30 & 16.44 & 192/192/192 \\
4096 & 48 & 34.7 & 436.6 & 1522.2 & 12.60 & 43.93 & 192/192/192 \\
\bottomrule
\end{tabular}
\end{table}
\FloatBarrier
On these six sizes, Dual is $2.55$--$12.60$ times faster than RSD and $4.61$--$43.93$ times faster than R--TR2.
All $3456$ single-target solves terminate normally and pass the accuracy checks.
The largest values of $r_{\rm F}$, $r_{\rm G}$, and $\delta$ are $3.76\times10^{-9}$, $9.92\times10^{-10}$, and $2.88\times10^{-9}$, respectively.
As shown in the top row of Figure~\ref{fig:numerical-comparison}, Dual's time grows more slowly than those of the two Riemannian methods over the tested values of $k$, so the speedup increases at each fixed $n$.

Table~\ref{tab:numerical-batch} compares the mean total times for sequential and batched solves with Dual.

\begin{table}[htbp]
\centering\small
\setlength{\tabcolsep}{5pt}
\caption{Mean total times in seconds for the same $64$ targets. The last column gives the numbers of successful sequential/batched runs, each out of $3$.}\label{tab:numerical-batch}
\begin{tabular}{rrrrrr}
\toprule
$n$ & $k$ & Sequential & Batched & Speedup & Successes\\
\midrule
2048 & 16 & 1.433 & 0.061 & 23.53 & 3/3 \\
2048 & 32 & 1.689 & 0.148 & 11.39 & 3/3 \\
2048 & 48 & 2.286 & 0.306 & 7.46 & 3/3 \\
4096 & 16 & 1.564 & 0.101 & 15.46 & 3/3 \\
4096 & 32 & 1.767 & 0.263 & 6.71 & 3/3 \\
4096 & 48 & 2.497 & 0.724 & 3.45 & 3/3 \\
\bottomrule
\end{tabular}
\end{table}
All $36$ runs pass the accuracy checks for every output.
Batching reduces the total solve time by a factor of $3.45$--$23.53$.
In the bottom row of Figure~\ref{fig:numerical-comparison}, the relative separation between the batched and sequential curves narrows as $k$ increases: batching provides less additional speedup as the work per problem grows.
One possible explanation is that larger individual problems make fuller use of GPU parallelism, leaving less room for batching to improve resource utilization.
\label{end:numerical}
\FloatBarrier

\section*{Conclusions}\phantomsection\label{sec:discussion}
\addcontentsline{toc}{section}{Conclusions}
We establish the exact threshold $n\ge2k-1$ for strong duality with arbitrary targets in the nearest rectangular symplectic matrix problem and clarify its connection to global optimality.
Within this range, the exact minimum-curvature formula explains geometrically why every second-order stationary point is a global minimizer, while the global objective error bounds quantify the suboptimality of feasible approximate solutions.
The extension to Hermitian positive definite row weights shows that universal strong duality is governed by the complex multiplicity of the smallest eigenvalue.
The unweighted dimension threshold follows from this criterion when the weight is the identity.

The dual formulation also yields a numerical method for the nearest rectangular symplectic matrix problem.
Across the six tested sizes and under the same accuracy requirements, its mean solve time is shorter than those of both the Riemannian first-order method and R--TR2 used for comparison.
For the same targets, batching the dual method further reduces the total time relative to sequential solves.
These experiments complement the global optimality analysis by showing that the same dual formulation can also improve computational efficiency for this class of matrix nearness problems.
\label{end:conclusion}

\section*{Declaration of AI-assisted work}
ChatGPT was used for language editing of the manuscript and to assist with code writing.

\section*{Acknowledgements}
This work was supported by the National Natural Science Foundation of China (Grant No.\ 92470119).
Yulin Xie acknowledges financial support from the China Scholarship Council (CSC).

\Needspace{5\baselineskip}
\phantomsection\label{sec:references}
\interlinepenalty=10000

\label{end:references}
\end{document}